\RequirePackage{fix-cm}
\documentclass[smallextended]{svjour3}       % onecolumn (second format)
\smartqed  % flush right qed marks, e.g. at end of proof
\usepackage{graphicx}
\usepackage{color}
\usepackage{amsmath,amssymb}
\usepackage{mathtools}
\usepackage{booktabs}
\usepackage{hyperref}

\newcommand{\Tind}{_{\mathrm{T}}}
\newcommand{\Rind}{_{\mathrm{R}}}
\DeclareMathOperator{\dtn}{\text{DtN}}
\newcommand{\diffq}[2]{\frac{\partial #1}{\partial #2}}

\newcommand{\colora}[1]{\textcolor{red}{#1}}
\newcommand{\colorb}[1]{\textcolor{blue}{#1}}

 \journalname{SN Partial Differential Equations and Applications}
\begin{document}

\title{Sweeping preconditioners for stratified media in the presence of reflections
}


\author{Janosch Preu\ss         \and
       Thorsten Hohage                \and 
       Christoph Lehrenfeld
}

\authorrunning{Preu\ss , Hohage,  Lehrenfeld} % if too long for running head

\institute{J. Preu\ss \at
              Max-Planck-Institut f\"ur Sonnensystemforschung, Justus-von-Liebig-Weg 3, 37077 G\"ottingen\\
              \email{preussj@mps.mpg.de}           %  \\
           \and
           T. Hohage \at
              Institut f\"ur Numerische und Angewandte Mathematik, Lotzestra{\ss}e 16-18, 37083 G\"ottingen and Max-Planck-Institut f\"ur Sonnensystemforschung, Justus-von-Liebig-Weg 3, 37077 G\"ottingen  \\ 
               \email{hohage@math.uni-goettingen.de}           %  \\
               \and       
            C. Lehrenfeld \at
              Institut f\"ur Numerische und Angewandte Mathematik, Lotzestra{\ss}e 16-18, 37083 G\"ottingen \\ 
               \email{lehrenfeld@math.uni-goettingen.de}           %  \\              
}

\date{Received: date / Accepted: date}

\maketitle

\begin{abstract}
  In this paper we consider sweeping preconditioners for time harmonic wave propagation in stratified media, especially in the presence of reflections. In the most famous class of sweeping preconditioners Dirichlet-to-Neumann operators for half-space problems are approximated through absorbing boundary conditions. In the presence of reflections absorbing boundary conditions are not accurate resulting in an unsatisfactory performance of these sweeping preconditioners.
  We explore the potential of using more accurate Dirichlet-to-Neumann operators within the sweep. To this end, we make use of the separability of the equation for the background model.
  While this improves the accuracy of the Dirichlet-to-Neumann operator, we find both from numerical tests and analytical arguments that it is very sensitive to perturbations in the presence of reflections.
  This implies that even if accurate approximations to Dirichlet-to-Neumann operators can be devised for a stratified medium, sweeping preconditioners are limited to very small perturbations.
  
\end{abstract}

\section{Introduction}\label{section:intro}

Time harmonic wave equations arise in various applications.
Their numerical discretization leads to large linear systems which are difficult to solve with classical iterative methods \cite{EZ12}.
Recently, sweeping preconditioners have emerged as a promising approach to overcome this problem \cite{GZ19}. 
Since the introduction of the moving perfectly matched layer (PML) preconditioner by Engquist and Ying \cite{EY11}, numerous impressive results and further developments
of this technique have been published.
We refer to \cite{GZ19} for a comprehensive review. \\

Unfortunately, the range of wave propagation problems in which sweeping preconditioners can be used is limited. 
We are not aware of any publication in which sweeping preconditioners have been successfully applied to media that 
contain strong resonant cavities. 
In fact, numerical experiments indicate that the established sweeping methods are not suitable for treating such problems, see e.g.
section 7.4 of \cite{DZ16} or section 10 of \cite{GZ19}.
Additionally, sweeping preconditioners require an absorbing boundary condition on at least one boundary of the domain at which the 
process of sweeping can be started. 
Since these assumptions are violated in many practically relevant problems, e.g. from global seismology, it is important to explore if these limitations of the 
sweeping technique can be overcome. \\ 

This paper investigates this question for the case of stratified media. 
Our problem setting differs significantly from the case of quasiperiodic Helmholtz transmission problems as recently considered in \cite{NPT20},
for instance, because it allows for a complete reflection of waves at the domain boundaries.
As a concrete example we consider a problem from \cite{IW95,JTCI08} in which spherical coordinates $\{r,\varphi,\theta\}$ are used.
Assuming axisymmetry of all fields in $\varphi$ the propagation of shear waves (SH-waves) $u$ between the core mantle boundary (CMB) at $r = R_{\text{CMB}}$ and the surface of the Earth at $r = R_{\oplus}$ in the frequency domain is described by the equation
\begin{equation}\label{eq: SH-waves frequency domain}
\mathcal{L}u = f , \text{ for } (r,\theta) \in \Omega := ( R_{\text{CMB} },R_{\oplus}) \times (0,\pi),
\end{equation}
where 
\begin{equation}\label{eq: diffop}
 \mathcal{L}u \coloneqq  - \rho \omega^2 u r^2 \sin^2(\theta) - \frac{\sin^2(\theta) }{r^2} \frac{\partial}{\partial r} \left( r^4 \mu \frac{\partial u}{\partial r} \right) - \frac{1}{\sin(\theta)} \frac{\partial}{\partial \theta} \left(  \sin^3(\theta) \mu \frac{\partial u}{\partial \theta} \right) 
\end{equation}
and  $f = f^{\text{s}}  r^2 \sin^2(\theta)$,
subject to boundary conditions $\mathcal{B}u = 0$.
The boundary operator  is defined piecewise on $\partial \Omega = \Gamma_{D} \cup \Gamma_{N}$ with 
$\Gamma_{D} = \{ \theta = 0 \} \cup \{ \theta = \pi \} $ and $\Gamma_{N }  = \{  r = R_{\text{CMB}}   \} \cup  \{  r = R_{\oplus}   \}     $   by
\begin{equation}\label{eq:SH-waves bry cond}
u = 0     \text{ on } \Gamma_{D}, \quad \frac{\partial u}{\partial r} =  0  \text{ on } \Gamma_{N}.
\end{equation}
Here, $\rho$ is the mass density, $\mu = \rho v_{SH}^2$ is the shear modulus and $\omega$ the frequency. The background coefficients for $v_{SH}$ and $\rho$ are provided by the spherically symmetric PREM model \cite{DA81}. 
In the appendix, cf. Section \ref{sec:deriv}, we give a derivation of the variational formulation that we base the finite element discretization on.
\\

Conventional sweeping preconditioners cannot be applied to this problem since absorbing boundary conditions are missing.
In this paper an extension of the sweeping preconditioner is presented which overcomes this limitation for the spherically 
symmetric background model.
It will further be investigated to which extent this preconditioner can then also be used for other models in which the coefficients $\rho$ and $v_ {SH}$ are small perturbations from the spherically symmetric case. 
This would be realistic, since 3D tomographic models of the Earth deviate only a few percent from the background model \cite{BB02}. \\

The remainder of this paper is structured as follows. 
In Section \ref{section: framework} we recall the general framework of sweeping preconditioners and the commonly employed 
moving PML approximation of the Dirichlet-to-Neumann (DtN) operator.
As a prototype of an improved approximation of the DtN operator we construct a DtN operator based on the separability of the background problem on the discrete level in Section \ref{section: tensor product DtN}.
The potential of this method is explored in Section \ref{section: numexp} with numerical experiments. 
Here it is observed that the DtN for the free surface boundary condition is very sensitive to perturbations. 
This observation is discussed in detail in Section \ref{sec:DtN sensitivity}.
We draw some conclusions regarding the applicability of sweeping preconditioners in the presence of reflections in Section \ref{section: concl}.

\section{General framework of sweeping preconditioners}\label{section: framework}

In \cite{GZ19} several sweeping methods have been described in the framework of the \textit{double sweep optimized Schwarz method} (DOSM). 
Here, we adapt DOSM to our specific setting. 
This includes the restriction to special cases, e.g. only non-overlapping domain decompositions and DtN transmission conditions are considered. 
The simplified version is more appropriate for the purpose of this paper since it allows to focus attention on the approximation of the DtN map.

\subsection{Double sweep optimized Schwarz method}\label{ssection: DOSM}

\begin{figure*}
\centering
  \includegraphics[width=\textwidth]{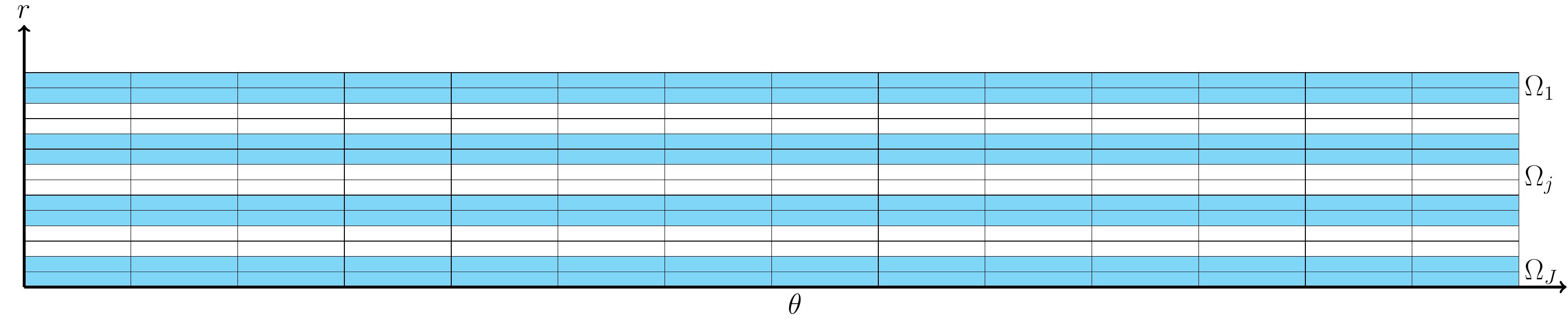}
\caption{Decomposition of the domain into layers.}
\label{fig:quad-mesh-tikz}      
\end{figure*}

Let $\Omega = \bigcup_{j=1}^{J} \Omega_j $ be a non-overlapping decomposition of the domain into horizontal layers $\Omega_j$, see Figure \ref{fig:quad-mesh-tikz}. 
To allow for an efficient solution of the subdomain problems the width of the layers in sweeping direction should be kept thin. 
Therefore, in the numerical examples below, we make the choice that one layer $\Omega_j$ contains only two finite elements in vertical direction ($r$) and $2 \cdot J$ elements in the horizontal direction ($\theta$) yielding a fixed aspect ratio for all finite elements. We denote by $\Gamma_{j,j \pm 1}   \coloneqq \partial \Omega_{j} \cap \partial \Omega_{j \pm 1} $ the interfaces between the layers and write $u_j$ for a function defined in layer $\Omega_j$. Based on these definitions we state the specialized sweeping algorithm. \\[1ex]
\emph{Forward sweep:} Given the last iterate $u_{j}^{(n-1)}$   in $\Omega_j, j = 1, \ldots,J$ solve successively for $j=1,\ldots,J-1$
\begin{align}
\mathcal{L}u_{j}^{(n-\frac{1}{2})}    &= f    && \text{in }  \Omega_j,  \nonumber \\  
\mathcal{B} u_{j}^{(n-\frac{1}{2})}  & = 0    && \text{on }  \partial \Omega \cap \partial \Omega_j,  \nonumber \\ 
r^4 \mu \frac{\partial u_{j}^{(n-\frac{1}{2})} }{\partial r} + \mathcal{P}_{j} u_{j}^{(n-\frac{1}{2})} & = r^4 \mu \frac{\partial u_{j-1}^{(n-\frac{1}{2})} }{\partial r} + \mathcal{P}_{j} u_{j-1}^{(n-\frac{1}{2})}      && \text{on }   \Gamma_{j,j - 1} \setminus \partial \Omega,              \nonumber  \\ 
  u_{j}^{(n-\frac{1}{2})} & = u_{j+1}^{(n-1)}   && \text{on }   \Gamma_{j,j + 1}.  \nonumber
\intertext{
                                                   \emph{Backward sweep:} Solve successively for $j = J, \ldots, 1$
}                                                   
\mathcal{L}u_{j}^{(n)}    &= f    && \text{in }  \Omega_j,  \nonumber \\  
\mathcal{B} u_{j}^{(n)}  & = 0    && \text{on }  \partial \Omega \cap \partial \Omega_j,  \nonumber \\ 
r^4 \mu \frac{\partial u_{j}^{(n)} }{\partial r} + \mathcal{P}_{j} u_{j}^{(n)} & = r^4 \mu \frac{\partial u_{j-1}^{(n-\frac{1}{2})} }{\partial r} + \mathcal{P}_{j} u_{j-1}^{(n-\frac{1}{2})}      && \text{on }   \Gamma_{j,j - 1},              \nonumber  \\ 
u_{j}^{(n)} & = u_{j+1}^{(n)}   && \text{on }   \Gamma_{j,j + 1}  \setminus \partial \Omega.  \nonumber
\end{align}
The transmission operator $\mathcal{P}_j$ is an (approximation) of the DtN map 
\begin{align}
\text{DtN}_j: & ~~g \mapsto - r^4 \mu \frac{\partial v}{\partial r},
% \end{equation}
\intertext{where $v$ solves}
% \begin{align}
\mathcal{L}v = 0 & ~~\text{ in } \Omega_{j}^{\text{ext}},   \nonumber \\ 
\mathcal{B}v = 0 & ~~\text{ on } \partial \Omega \cap \overline{\Omega_{j}^{\text{ext}} } ,   \\ 
v = g & ~~\text{ on } \Gamma_{j,j-1}, \nonumber
\end{align}
with $\Omega_{j}^{\text{ext}} = \bigcup_{i=1}^{j-1} \Omega_i$. 
If $\text{DtN}_{j}$ is well defined for $j=2,\ldots,J$ and the original problem and subdomain problems are uniquely solvable then DOSM  converges in one double sweep to the exact solution for $\mathcal{P}_j =\text{DtN}_{j}$  \cite{GZ19}. 
In practice, using the exact DtN map as a transmission operator is computationally too expensive. Therefore, $\mathcal{P}_j $ is chosen as an 
approximation of the DtN. The algorithm is then usually used as a preconditioner for GMRES with initial guess $u_{j}^{(0)} = 0, j = 1, \ldots,J$.

\subsection{Moving PML approximation of the DtN}\label{ssection: moving PML approx of DtN}

The problem for calculating the exact $\text{DtN}_j$ is posed on the whole exterior domain  $\Omega_{j}^{\text{ext}}$. Usually, it is assumed that on $\partial \Omega \cap \partial \Omega_1$ an absorbing boundary condition implemented by a PML is present. As the name ``moving PML'' suggests this PML is shifted closer to $\Omega_j$. This replaces the original problem posed on $\bigcup_{i=1}^{j-1} \Omega_i$ by a modified problem on the (usually smaller) domain $\Omega^{\text{PML}}_{j}$. In practice, the PML is usually started right at the coupling interface $\Gamma_{j,j-1}$ and $\Omega^{\text{PML}}_{j} = \Omega_{j-1}$. This leads to the operator 
\begin{align}
\text{DtN}_{j}^{\text{PML}}: & ~~g \mapsto - r^4 \tilde{\mu} \frac{\partial v}{\partial r},
% \end{equation}
\intertext{where $v$ solves}
% \begin{align}
\tilde{\mathcal{L}}v = 0 & ~~ \text{ in } \Omega_{j}^{\text{PML}}, \nonumber   \\ 
\tilde{\mathcal{B}}v = 0 & ~~ \text{ on } \partial \Omega_{j}^{\text{PML}} \setminus \partial \Omega_j,   \label{eq:L moving PML} \\ 
v = g & ~~ \text{ on } \Gamma_{j,j-1}. \nonumber
\end{align}
Here, the original differential operators $\mathcal{L}$, $\mathcal{B}$ have been replaced by modified versions $\tilde{\mathcal{L}}$, $\tilde{\mathcal{B}}$ due to the complex scaling applied 
in the PML region.

\section{A tensor product approximation of the DtN}\label{section: tensor product DtN}

In this section a new approach to approximate the DtN for tensor product discretizations will be presented. Let us note that the purpose of this approach is primarily to explore the potential of sweeping preconditioners under the assumption of accurate DtN approximations in what follows.

Let us fix an interface $\Gamma_{j,j-1}$ at $r = R_j$ on which the $\text{DtN}_{j}$ shall be computed. 
Let $W_{h}^{R_j}  \subset H^{1}\left( (R_j,R_{\oplus}) \times (0, \pi) \right)$ be the finite element space in  $\Omega_{j}^{\text{ext}}$. 
A tensor product discretization is assumed. 
Hence, $W_{h}^{R_j}  = V_{h}^{R_j} \otimes V_{h}^{\theta}$, where $ V_{h}^{R_j}  \subset H^{1}\left( (R_j,R_{\oplus})\right)$ and $V_{h}^{\theta} \subset H^{1}_{0}\left( (0,\pi) \right)$ are one-dimensional finite element spaces. 
In particular, $w_h \in W_{h}^{R_j} $ is a linear combination of terms of the form $R_{h}(r) \vartheta_{h}(\theta)$ with $R_h \in V_{h}^{R_j} $ and $ \vartheta_h \in V_{h}^{\theta}$.\\
Given Dirichlet data $g_h \in V_{h}^{\theta}$ the DtN in the finite element setting is computed as follows. 
First find $w_h \in W_{h}^{R_j}$ with $w_h(r=R_j,\cdot) = g_h$ such that 
\begin{align}\label{eq:dtn_2D}
a_{R_j}(w_h,v_h) & \coloneqq  \int\limits_{R_{j}}^{R_{\oplus}} \!\!  \int\limits_{0}^{\pi} \!\! \left( \!   - \rho r^2 \omega^2  w_h v_h \!+\!  r^2 \mu \frac{\partial w_h}{\partial r}  \frac{\partial v_h}{\partial r} \!+\!   \mu \frac{\partial w_h}{\partial \theta} \frac{\partial v_h}{\partial \theta}   \! \right) \!r^2 \sin^3(\theta)  d \theta dr  \nonumber \\ 
 & = 0 \text{ for all  } v_h \in  \{ v_h \in W_{h}^{R_j} \mid v_h(r=R_j,\cdot) = 0 \}. 
\end{align} 
Then $\text{DtN}_{j}(g_h) = -(R_{j})^4 \mu(R_j) \partial_r w_h(R_j,\cdot)$.  \\
This problem will be solved in two steps.
 In section \ref{ssection: tensor product DtN on disc-eigenfct} we treat the case in which $g_h$ is a discrete eigenfunction of the weighted Laplacian on $\Gamma_{j,j-1}$. 
 The action of the DtN on such data is given by multiplication with a number which can be computed by solving an ODE. 
 In other words, the DtN is diagonal in the basis of the discrete eigenfunctions. 
 This allows to treat the case of general Dirichlet data in section \ref{ssection: tensor product DtN on general data} by a simple change of basis.

\subsection{The DtN applied to a discrete eigenfunction}\label{ssection: tensor product DtN on disc-eigenfct}

Let $\psi^{\ell}(\theta) \in V_{h}^{\theta}$ for $\ell = 1, \ldots, L$ denote the discrete eigenfunctions of the discretized Laplacian on $\Gamma_{j,j-1}$ with eigenvalue $\lambda^{\ell}$, i.e. 
\begin{equation}\label{eq: discrete eigenfunctions}
\int\limits_{0}^{\pi} \frac{\partial \psi^{\ell }}{\partial \theta}  \frac{\partial \vartheta_h}{\partial \theta} \sin^3(\theta) ~d \theta  = \lambda^{\ell}
\int\limits_{0}^{\pi}  \psi^{\ell }  \vartheta_h  \sin^3(\theta) ~d \theta 
\end{equation}
for all $\vartheta_h \in V_{h}^{\theta}$. 

The next proposition shows that the $\psi^{\ell}$ are also eigenfunctions of the discretized DtN map.

\begin{proposition}\label{prop: DtN on discrete eigenfunction}
Let $u^{\ell}_{h} \in V_{h}^{R_j} $ with $u^{\ell}_{h}(R_j) = 1$ be the solution to 

\begin{equation}\label{eq:ODE for discrete modes}
\int\limits_{R_j}^{R_{\oplus}} \left( -\rho r^2 \omega^2 u^{\ell}_{h} R_h + r^2 \mu \frac{\partial u^{\ell}_{h}}{\partial r}  \frac{\partial R_{h}}{\partial r} + \lambda_{\ell} \mu u^{\ell}_{h} R_h \right) r^2 ~d r  = 0. 
\end{equation} 
for all $R_h \in  \{v_h \in V_{h}^{R_j} \mid v_h(R_j) = 0 \}$.

Then $DtN_{j}(\psi^{\ell}) = -(R_j)^4 \mu(R_j) \partial_r u^{\ell}_{h} (R_j) \psi^{\ell}$.
\end{proposition}

\begin{proof}
Define $w_h(r,\theta) = u^{\ell}_{h}(r) \psi^{\ell}(\theta) \in W_{h}^{R_j}$.
Note that $w_h(R_j,\theta) = \psi^{\ell}(\theta)$. So, if 
we can show that  $a_{R_j}(w_h,v_h) = 0$  for all $v_h \in  \{ v_h \in W_{h}^{R_j} \mid v_h(r=R_j,\cdot) = 0 \}$ then  $DtN_{j}(\psi^{\ell}) =   -(R_j)^4 \mu(R_j)  \partial_r w_h(R_j,\cdot)= -(R_j)^4 \mu(R_j) \partial_r u^{\ell}_{h}(R_j) \psi^{\ell}$ follows. 
Using (\ref{eq: discrete eigenfunctions}) for $v_h(r,\theta) = R_h(r) \vartheta_h(\theta)$ yields: 

\begin{align}
& a_{R_j}(w_h,v_h) = \int\limits_{R_{j}}^{R_{\oplus}} \left( -\rho r^2 \omega^2 u^{\ell}_{h} R_h + r^2 \mu \frac{\partial u^{\ell}_{h}}{\partial r}  \frac{\partial R_{h}}{\partial r} \right) r^2 ~d r   \int\limits_{0}^{\pi}   \psi^{\ell }  \vartheta_h  \sin^3(\theta) ~d \theta  \nonumber \\
 &+ \int\limits_{R_{j}}^{R_{\oplus}}  \mu u^{\ell}_{h} R_h  r^2 ~d r \int\limits_{0}^{\pi} \frac{\partial \psi^{\ell }}{\partial \theta}  \frac{\partial \vartheta_h}{\partial \theta} \sin^3(\theta) ~d \theta \nonumber \\
 & =  \int\limits_{R_{j}}^{R_{\oplus}} \left( -\rho r^2 \omega^2 u^{\ell}_{h} R_h + r^2 \mu \frac{\partial u^{\ell}_{h}}{\partial r}  \frac{\partial R_{h}}{\partial r} + \lambda_{\ell} \mu u^{\ell}_{h} R_h \right)  r^2 ~d r  \int\limits_{0}^{\pi}   \psi^{\ell }  \vartheta_h  \sin^3(\theta) ~d \theta \nonumber \\
 & = 0,
\end{align}

since $u^{\ell}_{h}$ solves (\ref{eq:ODE for discrete modes}).

\end{proof}

\begin{remark}
The proof of Proposition \ref{prop: DtN on discrete eigenfunction} uses as an essential ingredient that the product of the ODE solution and the discrete eigenfunction $u^{\ell}_{h}(r) \psi^{\ell}(\theta) $
is contained in the two dimensional finite element space $W_{h}^{R_j}$ employed for the solution of problem (\ref{eq:dtn_2D}).
This is ensured by letting the finite element space $V_{h}^{R_j}$ for the ODE (\ref{eq:ODE for discrete modes}) coincide with the first factor of the tensor product $W_{h}^{R_j}  = V_{h}^{R_j} \otimes V_{h}^{\theta}$. 
In other words, the ODE discretization is chosen as the restriction of the two dimensional discretization to the radial direction. The consequences of violating this requirement are 
discussed in Remark \ref{remark:numerical_perturb}.
\end{remark}

\subsection{General Dirichlet data }\label{ssection: tensor product DtN on general data}

To apply the DtN to general Dirichlet data $g_h \in V_{h}^{\theta}$ we express it in the eigenbasis $g_h = \sum\limits_{\ell=1}^{L}{ g_{\ell} \psi^{\ell}}$ and apply the DtN as 
\begin{equation}
\text{DtN}_j(g_h) = \sum\limits_{\ell=1}^{L}{ g_{\ell} \text{DtN}_j(\psi^{\ell})} 
                                      = - \sum\limits_{\ell=1}^{L}{  g_{\ell} (R_j)^4 \mu(R_j) \partial_r u^{\ell}_{h} (R_j) \psi^{\ell}   }.
\end{equation} 
Then we transform back to the finite element basis. 
The transformation to the discrete eigenbasis involves the solution of a dense linear system which is composed of the eigenvectors in the finite element basis.
%
% Possibly this inversion of a dense matrix can be avoided in some cases by designing analogues of infinite elements adapted to this setting, but the results of this paper may discourage from starting this effort.
%
In case of a uniform 2D mesh with periodic boundary conditions, the boundary mass and stiffness matrices are discrete block circulant matrices, and the transformation is given by the tensor (Kronecker) product of a small matrix and an FFT matrix, cf. \cite{rjasanow1994effective}. Possibly the inversion of dense matrices can also be avoided for more general grids in some cases by designing analogues of infinite elements adapted to this setting, but the results of this paper may discourage from starting this effort.

%\begin{remark}
%The continuous eigenfunctions of the (weighted) Laplace-Beltrami operator yield another basis in which the DtN diagonalizes. 
%This is for example utilized in the DtN FE method for the construction of absorbing boundary conditions \cite{G99}. 
%However, using the continuous eigenfunctions, which are smooth, in the context of preconditioning is not suitable since the Dirichlet data $g_h$ is merely continuous.
%Hence, a very large number of continuous eigenfunctions will be required to obtain an accurate approximation of  $g_h$ which makes this approach inefficient.
%Using the discrete eigenfunctions is more appropriate since they allow to express $g_h$ exactly (up to round-off errors) with finitely many coefficients.
%\end{remark}

\section{Numerical experiments}\label{section: numexp}

In this section numerical experiments for the model problem will be presented. 
First, the shortcomings of the moving PML approximation of the DtN are demonstrated.
Then we apply our new approximation to the spherically symmetric model and investigate its performance in case of small perturbations. \par 

All experiments are carried out using $H^1$-conforming finite elements and have been implemented in the finite element library \texttt{Netgen/NGSolve}, see \cite{JS97,JS14}.
Since the experiments feature piecewise smooth coefficients we have decided to use a medium finite element order of four.
This allows to benefit from the efficiency of higher order elements in the subdomains where the coefficients are smooth.
The discontinuities, which do not align with layer interfaces for the shear wave example, should be resolved through mesh refinement, i.e. by increasing the number of layers.
Scripts for reproducing the numerical results are provided at DOI: \href{https://doi.org/10.5281/zenodo.3886458}{10.5281/zenodo.3886458}.
This archive includes a \texttt{README} file which describes the structure of the code and gives detailed instructions on how to reproduce the presented results.

\subsection{Moving PML}\label{ssection: numexp moving PML}

The moving PML preconditioner relies on the assumption that the computational domain is truncated by an absorbing boundary condition in 
at least one direction and that the medium is free of large resonant cavities. 
The following two examples demonstrate that these assumptions are crucial.

\subsubsection{Academic example}\label{sssection:numexp mPML academic}

Consider the Helmholtz equation on the unit disk with discontinuous coefficients.
The bilinear form in polar coordinates is given by  
\begin{align}
a(u,v) = \int\limits_{0}^{1}  \int\limits_{0}^{2 \pi} \left(    - \frac{r \omega^2}{\rho c^2}  u v +  \frac{r}{\rho}  \frac{\partial u}{\partial r}  \frac{\partial v}{\partial r} +   \frac{1}{r \rho}\frac{\partial u}{\partial \theta} \frac{\partial v}{\partial \theta}    \right)   d \theta dr. \nonumber %\\
\end{align}
Periodic boundary conditions in $\theta$ are used. 
At $r =1.0$ an absorbing boundary condition implemented by a PML is set, while at $r= 0$ we
simply use natural boundary conditions. 
The geometrical setup is as shown in Figure \ref{fig:quad-mesh-tikz}. 
Similar to the experiments from \cite{GZ19} we let the wavespeed vary discontinuously between the layers where the strength of the discontinuity is given by a factor $\alpha$. 
To this end, we set on the first layer $c =1/ (1+\alpha / 2)$, on the next $c =1/ (1-\alpha / 2)$ then again $c =1/ (1+\alpha / 2)$ continuing in this fashion. 
The density is simply $\rho = 1$. \\

We apply the DOSM with moving PML approximation of the DtN to this problem. 
The number of subdomains is chosen to grow linearly with the wavenumber in order to counter the pollution effect \cite{BS97}.
The results are shown in Table \ref{tab:Moving PML-academic}. We also consider the case in which additional damping is added to the preconditioner\footnote{The original problem to which we apply the preconditioner is still the one without damping.}, i.e. the operator $\tilde{\mathcal{L}}$ in equation (\ref{eq:L moving PML}) describes a Helmholtz problem with complex frequency $ \omega + i \gamma$, where $\gamma = 1$. Adding additional damping to the preconditioner leads to nearly robust iteration numbers for $\alpha = 0$. However,  for both versions the iteration numbers increase drastically as the contrast $\alpha$ is increased. 
In further experiments, various choices for the damping parameter $\gamma$ have been considered.
However, a significant improvement in the case of high contrast could not be achieved.
 As a result, the preconditioner becomes completely inefficient in this setting. \\ 

\begin{table}
\caption{GMRES iteration numbers \colora{x} / \colorb{y} for moving PML approximation of the DtN for the academic example. For case \colorb{y} the preconditioner is constructed based on a damped problem with wavenumber $\omega \mapsto \omega + i \gamma$ for $\gamma = 1$. A random source which is set to zero in the PML layer is used as the right hand side. The dash means that the desired tolerance was not achieved within 200 iterations.}
\label{tab:Moving PML-academic}
\begin{center}
\begin{tabular}{cc@{\hspace*{0.6cm}}ccccc}
  \toprule
% \hline\noalign{\smallskip}
$\omega$ & $J$ & $\alpha =0$ & $\alpha=1/4$  & $\alpha = 1/2$  & $\alpha = 3/4$  & $\alpha = 1$     \\
  \midrule
    %     \noalign{\smallskip}\hline\noalign{\smallskip}
8 & 3  &  \colora{8} / \colorb{ 9} &  \colora{9} /  \colorb{10} & \colora{9} /  \colorb{11 }& \colora{8} /  \colorb{10 } & \colora{8} /   \colorb{10 }   \\
16 & 6  &  \colora{13} /  \colorb{11 } &  \colora{17} /  \colorb{17} & \colora{30} /  \colorb{32 } & \colora{39} /  \colorb{40 } & \colora{46} / \colorb{45 } \\
32 &  12 & \colora{15} /  \colorb{12 } & \colora{46} /  \colorb{37} & \colora{88} /  \colorb{93 }  & \colora{154} /  \colorb{162 } & \colora{-} /   \colorb{- } \\
64 & 24 &  \colora{19} /  \colorb{13 } & \colora{189} /  \colorb{185} & \colora{-} /  \colorb{ -} &  \colora{-}  /  \colorb{- } & \colora{-} /  \colorb{ -} \\
\bottomrule
  % \noalign{\smallskip}\hline
\end{tabular}
\end{center}
\end{table}

\subsubsection{SH-waves in frequency domain}\label{sssection:numexp mPML SH-wave}

In the academic example reflections generated by the strongly discontinuous coefficients lead to the method's breakdown.
Severe reflections can also be generated by boundary conditions as is the case for the model problem (\ref{eq: SH-waves frequency domain}) from seismology.
This can be demonstrated by comparing the GMRES iteration numbers for two different boundary conditions:
\begin{itemize}
\item In the first case, a PML is implemented at the Earth's surface.
\item In the second case, the PML is removed, which realizes the  physically desired free surface condition. 
\end{itemize}
The tolerance is set to $10^{-7}$. Two sources are considered: a Dirac and a random source. \\ 

The iteration numbers  for the case of a PML at the Earth's surface are shown in Table \ref{tab:Moving PML} in the columns `PML' . 
The wave field for $\omega = 2048$ with the Dirac source is shown in Figure \ref{fig:pt-source-omega2048-comp} (a). 
The waves generated from the point source bounce off from the CMB. Additional reflections are caused by the discontinuous coefficients and the Dirichlet boundary conditions at $\theta = 0$ and $\theta = \pi$.
Despite these difficulties, the preconditioner performs well. 
The iteration numbers grow only very mildly and the low tolerance of $10^{-7}$ is achieved for $\omega = 2048$ in 15 iterations for both sources. \\ 

The situation changes drastically when the PML at the Earth's surface is removed. 
The solver now has to capture a very complex wavefield created by additional reflections from the Earth's surface as Figure \ref{fig:pt-source-omega2048-comp} (b) shows.
The results for this case are given in Table \ref{tab:Moving PML} in the columns `free'.
The iteration numbers now grow drastically with increasing wavenumber. 
In Figure \ref{fig:res-l2err-movingPML-pt.} it can be seen that the residual stagnates for many iterations.
GMRES first has to filter out the reflections until a convergence to the solution occurs.
This renders the moving PML preconditioner unsuitable for our purpose.
Let us note that in additional computational studies we also tried adding damping to the preconditioning problem which however did not improve the situation.

\begin{figure*}

\centering
  \begin{tabular}{@{}c@{}}
     \includegraphics[width=\textwidth]{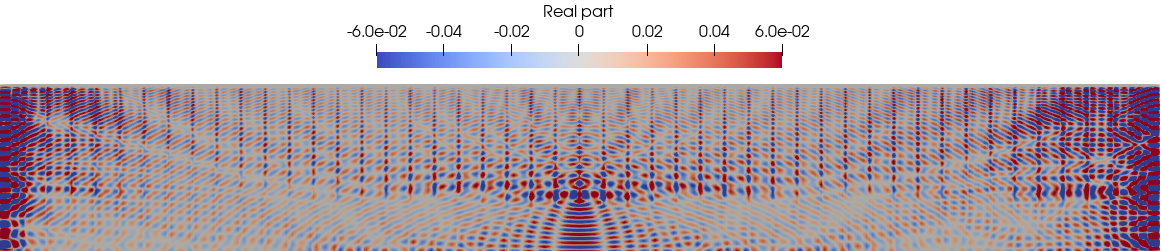} \\[\abovecaptionskip]
    \small (a) with PML.
  \end{tabular}

  \vspace{\floatsep}

  \begin{tabular}{@{}c@{}}
      \includegraphics[width=\textwidth]{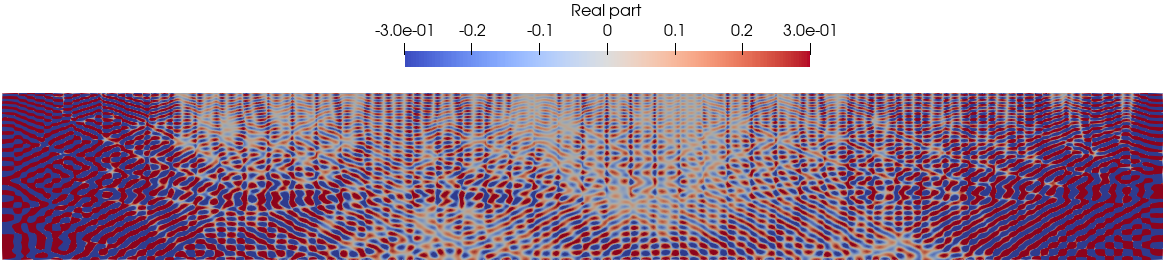} \\[\abovecaptionskip]
    \small (b) free surface.
  \end{tabular}
  
 \caption{ The real part of the solution for $\omega = 2048$ with a Dirac source. A PML was applied at the Earth's surface for (a) while (b) uses a free surface condition. }
\label{fig:pt-source-omega2048-comp}      
\end{figure*}

\begin{figure}
  \includegraphics[scale = 0.5]{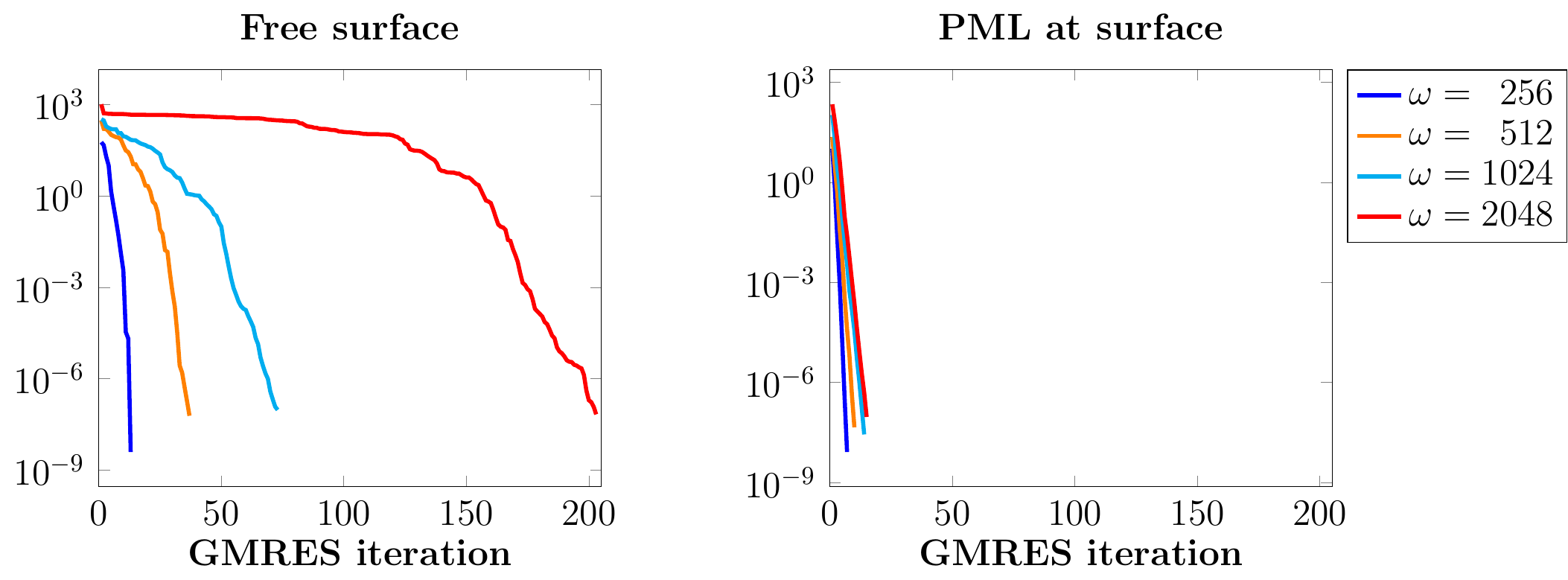}
\caption{GMRES residuals for moving PML approximation of the DtN with a Dirac source. On the left a free surface boundary condition is used at the Earth's surface while on the right a PML is implemented.}
\label{fig:res-l2err-movingPML-pt.}       
\end{figure}

\begin{table}
\caption{Influence of the boundary condition on GMRES iteration numbers for moving PML approximation of the DtN. The case denoted by 'PML' uses a radiation condition at the Earth's surface implemented by a PML while 'free' sets the free surface boundary condition. Two different sources are considered: a Dirac source (left columns) and a random source which is set to zero in the PML layer (right columns). }
\label{tab:Moving PML}
\begin{center}
\begin{tabular}{cccccc}
  % \hline\noalign{\smallskip}
  \toprule
  && \multicolumn{2}{c}{Dirac source}& \multicolumn{2}{c}{random source}\\
$\omega$ & $J$ &  \# iter (PML)  &  \# iter (free)  &  \# iter (PML)  &  \# iter (free)   \\
\midrule
  % \noalign{\smallskip}\hline\noalign{\smallskip}
256 & 3 & 7  & 13 & 7   & 18  \\
512  & 6 &  10 & 37 &  9 & 53  \\
1024 & 12 & 14  & 73  & 13  & 119 \\
  2048 & 24 & 15  & 203  & 15  & 307 \\
  \bottomrule
% \noalign{\smallskip}\hline
\end{tabular}
\end{center}
\end{table}

\subsection{Tensor product DtN  }\label{ssection: numexp tensor product}

We tackle problem (\ref{eq: SH-waves frequency domain}) with the free surface boundary condition using a tensor product discretization of the DtN as described in Section \ref{section: tensor product DtN}.
The $L^2$-error after one application of the preconditioner with respect to a solution computed with a direct solver is shown 
in Table \ref{tab:l2error tensor product}.
Apparently, the preconditioner can be used as a direct solver for the considered cases.
The growth of the error as the number of subdomains $J$ increases stems from a loss of precision in the eigenvalue equation \ref{eq: discrete eigenfunctions}, 
which occurs because we currently compute all eigenpairs explicitly. 
A practical implementation of our method would try to avoid such explicit computations by the techniques mentioned in Section \ref{ssection: tensor product DtN on general data}.
Analogous results are obtained for the academic example from Section \ref{sssection:numexp mPML academic}.
This confirms the theoretical result derived in Proposition \ref{prop: DtN on discrete eigenfunction}. 

\begin{table}
\caption{Relative $L^2$-error after one application of DOSM with tensor product DtN transmission conditions for a tensor product discretization.  Two different sources are considered: a Dirac source (left column) and a random source (right column). }
\label{tab:l2error tensor product}
\begin{center}
\begin{tabular}{cccc}
  \toprule
% \hline\noalign{\smallskip}
  $\omega$  &$J$ & error (Dirac) & error (random)  \\
  \midrule
% \noalign{\smallskip}\hline\noalign{\smallskip}
256 &       3  &   $9.76 \cdot 10^{-13}$    &    $7.70 \cdot 10^{-13}$     \\
512 &        6 & $3.76 \cdot 10^{-11}$       &  $6.94 \cdot 10^{-09}$          \\
1024 &   12 &   $4.46 \cdot 10^{-10}$     &      $9.36\cdot 10^{-10}$      \\
  2048 &  24 & $3.42 \cdot 10^{-08}$      &  $5.38 \cdot 10^{-07}$      \\
  \bottomrule
% \noalign{\smallskip}\hline
\end{tabular}
\end{center}
\end{table}

\subsubsection{Sound speed perturbations}\label{sssection:tensor product perturb}

So far only the radially symmetric shear velocity $v_{\text{PREM}}(r)$ of PREM has been considered.
Let us now introduce a velocity perturbation 
\begin{equation}\label{eq:perturb}
v_{\text{pert}}(r,\theta) = \cos(r \theta) \sin(r \theta).
\end{equation}
Setting $v(r,\theta) = v_{\text{PREM}}(r) \left( 1 + \varepsilon v_{\text{pert}}(r,\theta)  \right)$ results in a relative perturbation 
of strength $\varepsilon$. 
In the following, the tensor product DtN based on $v_{\text{PREM}}$ is employed to precondition the system for $v(r,\theta)$. \\ 

In Table \ref{tab:tensor product DtN sound speed perturbation} the iteration numbers obtained with the tensor product DtN are 
compared to the moving PML\footnote{The results for the moving PML with $\varepsilon = 0$ are slightly different compared to Table \ref{tab:Moving PML} because the random source is not set to zero in the uppermost layer.}  based approximation. 
The moving PML yields iteration numbers which are essentially independent of the strength of the perturbation since it approximates the DtN
based on the perturbed shear velocity. 
The tensor product DtN instead is obtained from the separable background velocity. 
Thus, iteration numbers grow with the strength of the perturbation. \par

The performance of both approaches is highly dependent on the boundary conditions imposed at the Earth's surface. 
The tensor product DtN is very robust against perturbations for the PML boundary condition. 
As a result, it outperforms the moving PML method for perturbations up to $2 \%$.
When the PML is replaced by a free surface boundary condition then the DtN apparently becomes very sensitive to perturbations.
Hence, the tensor product DtN only yields acceptable iteration numbers in the high frequency regime for perturbations smaller than $0.1 \%$.
\begin{remark}\label{remark:numerical_perturb}
  Even in cases where the model is perfectly captured in the solution of the exterior problem in terms of the data, i.e. assuming no perturbations in the model coefficients, one often needs to deal with numerical perturbations. For instance, we considered problems without data perturbation where, however, adaptive meshes that violate the tensor product structure or for the solution of \eqref{eq:ODE for discrete modes} different ODE solvers have been used. These numerical perturbations resulted in the same effect, a dramatic amplification thereof, rendering the approach practically useless in the presence of reflections.
\end{remark}

The tensor product DtN approximation considered above realizes an exact solution of the exterior problem in the case of no perturbation of the background model. The missing robustness of this approach is not due to the tensor product construction, but rather due to the high sensitivity of the exact DtN operator itself. The investigation of this effect will be the subject of the subsequent section. 

\begin{table}
\caption{Comparison of GMRES iteration numbers \colora{x} / \colorb{y}: The numbers \colora{x} are obtained when the tensor product DtN based on the background model is used to set up a preconditioner for the system 
with a perturbed shear velocity. For the case  \colorb{y} the moving PML approximation of the DtN is employed.  The strength of the perturbation $\varepsilon$ is given in $\%$. A random source is used for the right hand side (set to zero where PML is present). The dash ' -'  means that the desired tolerance of $10^{-7}$ was not reached after 1000 iterations. }
\label{tab:tensor product DtN sound speed perturbation}      
\center{Free surface condition at Earth's surface} \\[0.5ex]
\begin{small}
\begin{tabular}{rrr@{/}rr@{/}rr@{/}rr@{/}rr@{/}rr@{/}rr@{/}r}
    %     \hline\noalign{\smallskip}
  \toprule
  $\omega$  &$J$& \multicolumn{2}{c}{$0 \%$} &
                  \multicolumn{2}{c}{$0.0625 \%$} &
                  \multicolumn{2}{c}{$0.125 \%$} &
                  \multicolumn{2}{c}{$0.25 \%$} &
                  \multicolumn{2}{c}{$0.5 \%$} & 
                  \multicolumn{2}{c}{$1 \%$}  &
                  \multicolumn{2}{c}{ $2 \%$}  \\
  \midrule
% \noalign{\smallskip}\hline\noalign{\smallskip}
256 &       3  & \colora{1}&\colorb{~18}   & \colora{4}&\colorb{~18}   & \colora{4}&\colorb{~18} &  \colora{6}&\colorb{~18}    & \colora{7}&\colorb{~18} & \colora{8}&\colorb{~18} & \colora{11}&\colorb{~18} \\
512 &        6 &  \colora{1}&\colorb{~53}   & \colora{9}&\colorb{~53}     &  \colora{12}&\colorb{~53} & \colora{17}&\colorb{~53}    & \colora{26}&\colorb{~53} &  \colora{39}&\colorb{~53}  &  \colora{62}&\colorb{~53}  \\
1024 &   12 & \colora{2}&\colorb{119}   & \colora{16}&\colorb{120}    & \colora{23}&\colorb{120}    &   \colora{35}&\colorb{122}    &  \colora{53}&\colorb{122} &    \colora{84}&\colorb{118}    &  \colora{158}&\colorb{123}   \\
  2048 &  24 &  \colora{2}&\colorb{292}   & \colora{38}&\colorb{292}   & \colora{102}&\colorb{293}  &  \colora{215}&\colorb{ 274} & \colora{422}&\colorb{291}  &  \colora{735}&\colorb{238} & \colora{-}&\colorb{ 274} \\
  \bottomrule
% \noalign{\smallskip}\hline
\end{tabular}
\end{small}
\center{PML at Earth's surface} \\[0.5ex]
\begin{small}
\begin{tabular}{rrr@{/}rr@{/}rr@{/}rr@{/}rr@{/}rr@{/}rr@{/}rr@{/}r}
    %     \hline\noalign{\smallskip}
  \toprule
  $\omega$  &$J$& \multicolumn{2}{c}{$0 \%$} &
                  \multicolumn{2}{c}{$0.0625 \%$} &
                  \multicolumn{2}{c}{$0.125 \%$} &
                  \multicolumn{2}{c}{$0.25 \%$} &
                  \multicolumn{2}{c}{$0.5 \%$} & 
                  \multicolumn{2}{c}{$1 \%$}  &
                  \multicolumn{2}{c}{ $2 \%$}  &
                  \multicolumn{2}{c}{ $4 \%$}  \\
  \midrule
256 &       3   & \colora{1}&\colorb{7}   & \colora{3}&\colorb{7}   & \colora{3}&\colorb{7} &  \colora{3}&\colorb{7}    & \colora{3}&\colorb{7} & \colora{4}&\colorb{7} & \colora{4}&\colorb{7}  & \colora{5}&\colorb{7}        \\
512 &        6   & \colora{1}&\colorb{9}   & \colora{3}&\colorb{9}   & \colora{4}&\colorb{9} &  \colora{4}&\colorb{9}    & \colora{4}&\colorb{9} & \colora{5}&\colorb{9} & \colora{6}&\colorb{9}       & \colora{7}&\colorb{9}               \\
1024 &   12   & \colora{1}&\colorb{13}   & \colora{4}&\colorb{14}   & \colora{5}&\colorb{14} &  \colora{6}&\colorb{14}    & \colora{7}&\colorb{14} & \colora{8}&\colorb{14} & \colora{10}&\colorb{14}      & \colora{14}&\colorb{13}                 \\
  2048 &  24   & \colora{2}&\colorb{15}   & \colora{5}&\colorb{15}   & \colora{6}&\colorb{15} &  \colora{7}&\colorb{15}    & \colora{8}&\colorb{16} & \colora{11}&\colorb{16} & \colora{15}&\colorb{16}          & \colora{22}&\colorb{16}             \\
\bottomrule
\end{tabular}
\end{small}
\end{table}

\section{On the sensitivity of DtN operators}\label{sec:DtN sensitivity}

The experiments from the previous section suggest that the DtN for the free surface boundary condition is much more sensitive to perturbations 
than for an absorbing boundary condition.
Performing a mode-by-mode analysis, see section \ref{ssec:Modal analysis}, confirms this and reveals that the issue is already observed in one dimension.
This leads to an analytical sensitivity analysis of scattering problems on the half line via Riccati equations in Section \ref{ssec:analysis Riccati}. Additional 
illustrations for this case are provided in Section \ref{ssec:Illustration 1D}.

\subsection{Modal analysis}\label{ssec:Modal analysis}

If in the study above in Section \ref{sssection:tensor product perturb} we replace the perturbation model
\eqref{eq:perturb} with 
only linear velocity perturbations, i.e. $v(r) = v_{\text{PREM}}(r) \left( 1 + \varepsilon   \right)$ with constant $\varepsilon \geq 0$,
then the perturbed problem is separable as well. Denote by  $\dtn_{j}(\psi^{\ell},\varepsilon)$ the DtN numbers as computed in Proposition \ref{prop: DtN on discrete eigenfunction} for perturbation $\varepsilon$. We fix the interface $j=J-1$ closest to the CMB. In Figure \ref{fig:numerical-DtN-error} the relative DtN error 
\begin{equation}\label{eq:numerical DtN error}
  \frac{ \vert  \dtn_{j}(\psi^{\ell},0) -  \dtn_{j}(\psi^{\ell},\varepsilon) \vert }{   \vert  \dtn_{j}(\psi^{\ell},0)   \vert }
\end{equation}
with $\varepsilon \approx  3.9 \cdot 10^{-5} $ is shown for $\omega \in  \{512,1024,2048\}$. 
The same mesh, in particular the same discrete modes, have been used for all $\omega$. \par 
The relative error on the guided modes $\vert \lambda_{\ell} \vert \lesssim \omega^2$ for the free surface boundary condition is highly oscillatory and 
even in the best case almost two orders of magnitude larger than for the PML boundary condition. 
It also grows linearly in $\omega$ and $\varepsilon$. 
The first statement can be seen in the Figure while the second was observed in other experiments not shown here. 
These findings show that the DtN for the free surface boundary condition is indeed highly sensitive to perturbations which explains 
the poor performance of the tensor product DtN based on the background model as observed in Section \ref{sssection:tensor product perturb}.

\begin{figure*}
\centering
  \includegraphics[width=\textwidth]{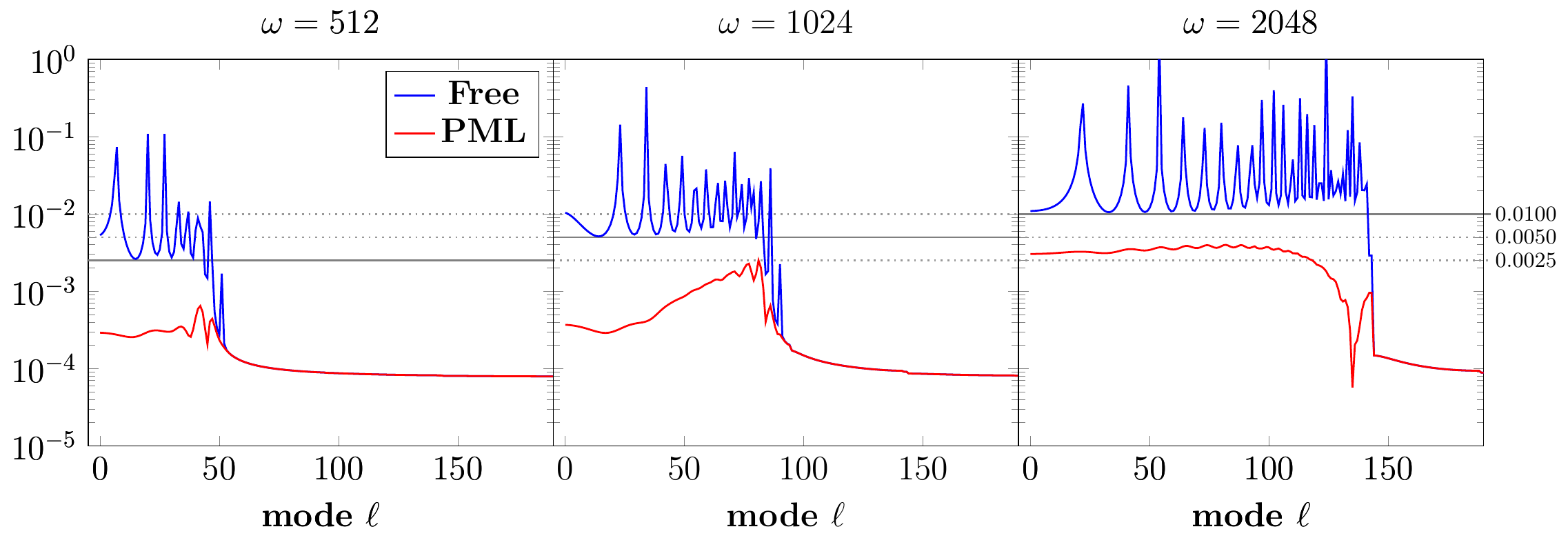}
\caption{ The relative DtN error on the discrete eigenmodes as computed in (\ref{eq:numerical DtN error}) for a linear velocity perturbation of fixed strength $\varepsilon$. The blue line is for a free surface boundary condition at the Earth's surface while the red line is obtained with an absorbing boundary condition implemented by a PML. The modes $\ell$ have been ordered by increasing magnitude of $\lambda_{\ell}$.  }
\label{fig:numerical-DtN-error}      
\end{figure*}

\subsection{Analysis via Riccati equations}\label{ssec:analysis Riccati}
The high sensitivity already occurs in one dimension as demonstrated by the modal analysis. 
Hence, to get to the core of the problem, we will analyze the one dimensional scattering problems with transparent boundary conditions 
\begin{subequations}
\begin{equation}\label{eq:scat_halfline_eps} \tag{\theequation-T}
\begin{cases}
-u\Tind'' - \omega^2(1+\varepsilon(x)) u\Tind = 0& \mbox{on } (0,a)\\
u\Tind(0)=1\\
u\Tind'(a) = i \omega u\Tind(a)
\end{cases}
\end{equation} 
and reflecting boundary conditions 
\begin{equation}\label{eq:scat_interv_eps} \tag{\theequation-R}
\begin{cases}
-u\Rind'' - \omega^2(1+\varepsilon(x)) u\Rind = 0& \mbox{on } (0,a)\\
u\Rind(0)=1\\
u\Rind(a) = 0
\end{cases}
\end{equation}
\end{subequations}
for $a>0$ and  $\Im  \omega \geq 0$. \par 
 
Let us consider the functions 
\[
v_{\mathrm{B}}(x):= \frac{u_{\mathrm{B}}'(x)}{u_{\mathrm{B}}(x)},
\qquad \mathrm{B}\in \{\mathrm{T},\mathrm{R}\},
\]
which may be interpreted as  ``local DtN numbers''. It is well-known and easy to check that these functions satisfy the Riccati equation
\begin{equation}\label{eq:riccati}
v'(x) = -\omega^2(1+\varepsilon(x)) - v(x)^2\,.
\end{equation} 
Moreover, we have the initial conditions
\begin{equation}
v_{\mathrm{T}}(a) = i \omega,\qquad v_{\mathrm{R}}(a)=\infty.
\end{equation}
For simplicity, let us switch to the perturbation $E(x):=\omega^2 \varepsilon(x)$. The Fr\'{e}chet derivatives
\[
k_{\mathrm{B}}:= \diffq{v_{\mathrm{B}}}{E}\tilde{E}
\]
(formally) satisfy the linear initial value problems 
\begin{equation}\label{eq:pert_ric}
k_{\mathrm{B}}' = - \tilde{E} -2 v_{\mathrm{B}} k_{\mathrm{B}},\qquad 
k_{\mathrm{B}}(a) = 0,
\end{equation}
which follows by differentiating (\ref{eq:riccati}) with respect to $E$ \footnote{As we only aim for a heuristic explanation rather then a theorem in this section, we do not go into the technicalities of justifying this statement rigorously in the presence of singularities of $v_B$.}.
Our aim is to compare the sizes of $k_{\mathrm{T}}(0)$ and $k_{\mathrm{R}}(0)$.

The solutions to the initial value problems \eqref{eq:pert_ric}
can be expressed in terms of the solutions
\[
k^y_{\mathrm{B}}(x):= \exp\left(\int_x^y 2v_{\mathrm{B}}(z)\,dz \right)
\]
to the homogeneous equations 
${k^y_{\mathrm{B}}}' =  -2 v_{\mathrm{B}} k_{\mathrm{B}}^y$ with initial conditions $k_{\mathrm{B}}^y(y)=1$ as 
\begin{equation}\label{eq:sol_w}
k_{\mathrm{B}}(x) = \int_x^a k^y_{\mathrm{B}}(x)\tilde{E}(y)\, dy\,.
\end{equation}
The main difference between transparent and reflecting boundary conditions is that 
$v_{\mathrm{T}}$ is typically close to purely imaginary (for $E=0$ it is 
identically $i \omega$) whereas $v_{\mathrm{R}}$ for real-valued $\omega$ and $E$ is real-valued and has singularities at zeros of $u_{\mathrm{R}}$. For $E=0$ we have 
$v_{\mathrm{R}}(x) = - \omega \cot(\omega(a-x))$. 
Therefore, $|k^y_{\mathrm{T}}|$ is close to 
$1$ whereas the kernel  $|k^y_{\mathrm{R}}|$ has singularities, and its modulus is often much larger than $1$. In view of \eqref{eq:sol_w}, 
this explains, why $|k_{\mathrm{R}}(0)|$ is typically 
much greater than $|k_{\mathrm{T}}(0)|$.  
This conclusion is consistent with the relative DtN error for the guided modes as shown in Figure \ref{fig:numerical-DtN-error}.
In contrast, $\vert \lambda_{\ell} \vert \gtrsim \omega^2$ leads to exponentially decaying modes and removes the singularities from the 
corresponding kernel.
In this case, $|k_{\mathrm{R}}(0)|$ and $|k_{\mathrm{T}}(0)|$ are of comparable size.

\subsection{Comparison of DtN numbers for constant perturbation}\label{ssec:Illustration 1D}

Let us illustrate the qualitative statements of the previous section by considering only 
constant perturbations $\varepsilon(x) = \varepsilon$ which allows us to compute the solutions to (\ref{eq:scat_halfline_eps}) and (\ref{eq:scat_interv_eps}). These problems become constant coefficient 
problems in the perturbed wavenumber $\omega_{\varepsilon} = \omega \sqrt{1+\varepsilon}$. The solutions for $\omega>0$ for
transparent boundary conditions
\begin{equation*}
u\Tind^{\varepsilon}(x) =  \frac{\sqrt{1+\varepsilon} \cos( \omega_{\varepsilon }(x-a) ) + i  \sin( \omega_{\varepsilon }(x-a) )   } { \sqrt{1+\varepsilon} \cos( \omega_{\varepsilon }a ) - i  \sin( \omega_{\varepsilon }a)  }
\end{equation*}
 and reflecting boundary conditions 
\begin{equation*}
u\Rind^{\varepsilon}(x) = \cos (\omega_{\varepsilon}x) - \cot(\omega_{\varepsilon}a) \sin( \omega_{\varepsilon}x)
\end{equation*} 
yield the perturbed DtN numbers 
\begin{align*}
\text{DtN}\Tind(\varepsilon):= & - (u\Tind^{\varepsilon})'(0) = - i \omega -  \omega \varepsilon \frac{ \sin(\omega_{\varepsilon } a) }{\sqrt{1+\varepsilon} \cos( \omega_{\varepsilon }a ) - i  \sin( \omega_{\varepsilon }a) },   \\ 
\dtn\Rind(\varepsilon):=& -(u\Rind^{\varepsilon})'(0) =  \omega_{\varepsilon} \cot(\omega_{\varepsilon}a). 
\end{align*}
Comparing this with the background profile gives the relative errors
\begin{align*}
  \Delta \Tind(\varepsilon,\omega) & :=  \vert \dtn \Tind(0)-\dtn \Tind(\varepsilon) \vert / \vert \dtn \Tind(0) \vert  \\
                                                                     & = \varepsilon  \frac{\vert \sin(\omega_{\varepsilon } a)  \vert  }{ \sqrt{ 1 + \varepsilon \cos( \omega_{\varepsilon } a)^2  }} \leq \varepsilon
\intertext{ and }   
  \Delta \Rind(\varepsilon,\omega) & :=  \vert \dtn\Rind(0)-\dtn\Rind(\varepsilon) \vert / \vert \dtn\Rind(0) \vert \\
  & = \frac{ \vert \cot(\omega a) - \sqrt{1+\varepsilon} \cot( \sqrt{1+\varepsilon} \omega a) \vert}{ \vert  \cot(\omega a)  \vert} = \varepsilon \left\vert -\frac{1}{2}+\frac{\omega a}{\sin(2\omega a)} \right\vert + O(\varepsilon^2). 
\end{align*}

In Figure \ref{fig:analytic-DtN-error} these functions are plotted for constant $\varepsilon =  10^{-3} $ and $a=1$. In this case, the relative error 
for the transparent boundary conditions can be bounded independently of $\omega$ while
the relative error for reflecting boundary conditions is much larger, highly oscillatory and grows like $\omega$. 

\begin{figure*}
\centering
  \includegraphics[width=.8\textwidth]{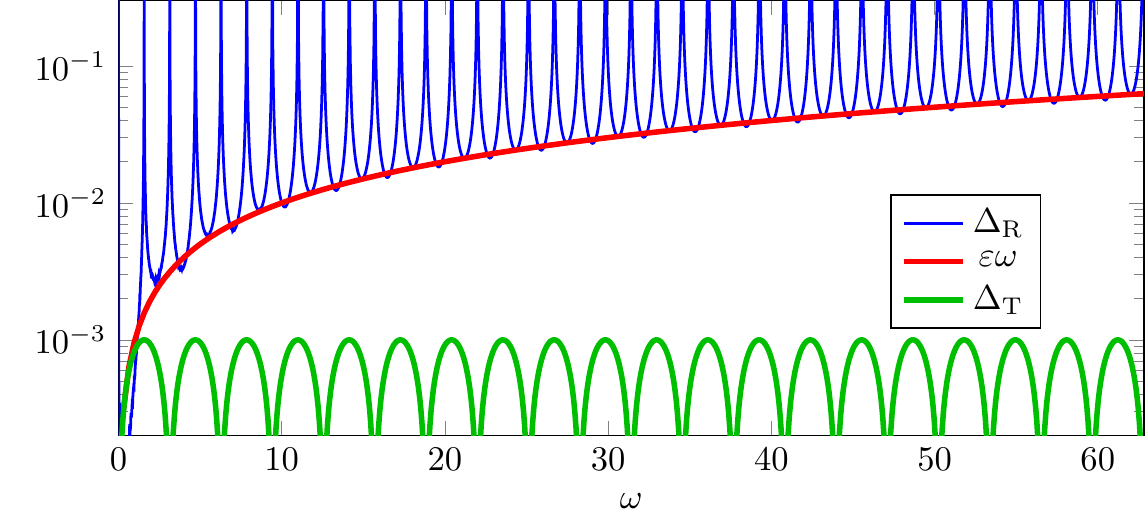}
\caption{ Comparison of analytic DtN errors for a constant perturbation $\varepsilon = 10^{-3} $ and $a=1$. Here, $\mathrm{R}$ and $\mathrm{T}$ denote reflecting
and transparent boundary conditions respectively.}
\label{fig:analytic-DtN-error}      
\end{figure*}

\section{Conclusion}\label{section: concl}

This paper investigates the potential of sweeping preconditioners for stratified media in absence of an absorbing boundary condition.
For such a problem the DtN cannot be reasonably approximated by a moving PML.
To resolve this issue, a tensor product discretization of the DtN, which is based on separability of the equation for the background model, has been introduced yielding a direct solver for the unperturbed background model.
Despite its perfect approximation of the DtN the applicability of the resulting sweeping preconditioner is limited due to a very high sensitivity of the DtN to perturbations. 
As a conclusion we can state that in the presence of reflections any sweeping preconditioner for wave propagation relying on an accurate, but not perfect, DtN approximation -- based on a tensor product structure or not -- is doomed to fail in practice.

%\begin{figure}
%  \includegraphics{example.eps}
%\end{figure}
%\begin{figure*}
 % \includegraphics[width=0.75\textwidth]{example.eps}
%\end{figure*}
%\begin{table}
%\begin{tabular}{lll}
%\hline\noalign{\smallskip}
%\end{tabular}
%\end{table}

\begin{acknowledgements}
  J. {Preu\ss}  is  a  member  of  the  International  Max  Planck Research School for Solar System Science at the University of G\"ottingen; \\[1ex]
  \textbf{Author contributions}
  T.H. initiated the research on sweeping preconditioners for problems with reflection. All authors designed and performed the research. Selection of the model and derivation of the discrete formulation of the SH-waves example is due to J.P.. All implementations and numerical computations were performed by J.P. using the finite element library NGSolve. J.P. drafted the paper and all authors contributed to the final manuscript.
\end{acknowledgements}

 \section*{Conflict of interest}
 The authors declare that they have no conflict of interest.

\appendix
\section{Derivation of the variational formulation}\label{sec:deriv}
In \cite{IW95} the equation for SH-waves in the time domain assuming axial symmetry is given as 
\begin{equation} \tag{A1}
\rho \frac{\partial^2 u }{\partial t^2} = f + \frac{1}{r^4} \frac{\partial}{\partial r }\left(  r^4 \mu \frac{\partial u}{\partial r} \right) + \frac{1}{r^2 \sin^3(\theta)} \frac{\partial}{\partial \theta} \left(  \sin^3(\theta) \mu \frac{\partial u}{\partial \theta} \right).
\end{equation}
Multiplying with $r^2 \sin^2(\theta)$ and transforming to the frequency domain we obtain 
\begin{equation} \tag{A2}
- \rho \omega^2 u r^2 \sin^2(\theta) - \frac{\sin^2(\theta) }{r^2} \frac{\partial}{\partial r} \left( r^4 \mu \frac{\partial u}{\partial r} \right) - \frac{1}{\sin(\theta)} \frac{\partial}{\partial \theta} \left(  \sin^3(\theta) \mu \frac{\partial u}{\partial \theta} \right) = f r^2 \sin^2(\theta).
\end{equation}
To obtain a variational formulation in spherical coordinates this equation has to be integrated taking the volume 
element $r ^2 \sin(\theta) dr d \theta d \varphi$ into account. Due to axisymmetry the integration over $\varphi$ only yields 
a constant factor, which can be omitted. We obtain the linear form: 
\begin{equation} \tag{A3}
f(v) =   \int\limits_{R_{ \text{CMB}}}^{R_{\oplus}}  \int\limits_{0}^{\pi} f(r,\theta) v r^4 \sin^3(\theta) ~ d \theta dr.
\end{equation}
To obtain the bilinear form we integrate by parts in $r$ and $\theta$. The boundary terms in $r$ vanish because of the free surface 
boundary condition. Since homogeneous Dirichlet boundary conditions are imposed at $\theta = 0$ and $\theta = \pi$ the boundary 
terms in $\theta$ vanish as well.
Let $H^1_0(\Omega)$ be the subspace of $H^1(\Omega)$ with homogeneous boundary values at $\theta = 0$ and $\theta = \pi$. The variational problem is: Find $u \in H^1_0(\Omega)$ so that $a(u,v) = f(v)$ for all $v \in H^1_0(\Omega)$ with 
\begin{align} 
a(u,v) & = \!\!\int\limits_{R_{ \text{CMB}}}^{R_{\oplus}} \!\!\! \int\limits_{0}^{\pi} \! \left( \!   - \rho \omega^2  u v r^4 \sin^3(\theta) - \frac{\partial}{\partial r} \! \left( r^4 \mu \frac{\partial u}{\partial r} \right) v \sin^3(\theta) - r^2  \frac{\partial}{\partial \theta} \left(  \sin^3(\theta) \mu \frac{\partial u}{\partial \theta} \right) \! v  \! \right)  \!d \theta dr  \nonumber \\  \tag{A4}
            & = \!\! \int\limits_{R_{ \text{CMB}}}^{R_{\oplus}} \!\!\!  \int\limits_{0}^{\pi} \! \left( \!   - \rho r^2 \omega^2  u v +  r^2 \mu \frac{\partial u}{\partial r}  \frac{\partial v}{\partial r} +   \mu \frac{\partial u}{\partial \theta} \frac{\partial v}{\partial \theta}    \right) r^2 \sin^3(\theta)  d \theta dr .
\end{align}

\bibliographystyle{spmpsci}      % mathematics and physical sciences
\bibliography{references}   % name your BibTeX data base

\end{document}